\documentclass[12pt]{amsart}
\usepackage{a4wide}
\usepackage[T1]{fontenc}
\usepackage{amssymb,amsmath,amsthm,latexsym}
\usepackage{mathrsfs}
\usepackage[usenames,dvipsnames]{color}
\usepackage{euscript}
\usepackage{graphicx}
\usepackage{mdwlist}
\usepackage{enumerate}
\usepackage{mathtools,dsfont,wasysym}
\usepackage{bbm}
\usepackage{stmaryrd}
\usepackage{centernot}
\usepackage{hyperref}
\hypersetup{colorlinks=true,  linkcolor=blue, citecolor=magenta, urlcolor=cyan}

\makeatletter
\@namedef{subjclassname@2020}{\textup{2020} Mathematics Subject Classification}
\makeatother

\newtheorem{theorem}{Theorem}[section]
\newtheorem{lemma}[theorem]{Lemma}
\newtheorem{corollary}[theorem]{Corollary}
\newtheorem{proposition}[theorem]{Proposition}
\newtheorem{claim}[theorem]{Claim}
\newtheorem{fact}[theorem]{Fact}
\newtheorem{problem}[theorem]{Problem}
\newcounter{maintheorem}

\newtheorem{mainth}[maintheorem]{Theorem}
\theoremstyle{remark}
\newtheorem{remark}[theorem]{Remark}
\theoremstyle{definition}
\newtheorem{definition}[theorem]{Definition}

\newtheorem{example}[theorem]{Example}
\numberwithin{equation}{section}
\makeatother

\newcommand{\N}{\mathbb{N}}
\newcommand{\e}{\varepsilon}
\newcommand{\p}{\varphi}

\newcommand{\nn}[1]{{\left\vert\kern-0.25ex\left\vert\kern-0.25ex\left\vert #1 
\right\vert\kern-0.25ex\right\vert\kern-0.25ex\right\vert}}
\newcommand{\cut}{\mathord{\upharpoonright}}
\renewcommand{\leq}{\leqslant}
\renewcommand{\geq}{\geqslant}
\DeclareMathOperator{\supp}{supp}
\DeclareMathOperator{\dens}{dens}
\DeclareMathOperator{\cf}{cf}
\DeclareMathOperator{\ims}{ims}
\DeclareMathOperator{\Lev}{Lev}
\DeclareMathOperator{\htte}{ht}

\DeclareMathOperator{\rk}{rk}

\newcommand{\Span}{\operatorname{span}}

\newcounter{smallromans}

\newenvironment{romanenumerate}
{\begin{list}{{\normalfont\textrm{(\roman{smallromans})}}}%
  {\usecounter{smallromans}\setlength{\itemindent}{0cm}%
   \setlength{\leftmargin}{5.5ex}\setlength{\labelwidth}{5.5ex}%
   \setlength{\topsep}{.5ex}\setlength{\partopsep}{.5ex}%
   \setlength{\itemsep}{0.1ex}}}%
{\end{list}}

\newcommand{\bone}{\mathbbm{1}}
\renewcommand\qedsymbol{$\blacksquare$} 

\begin{document}
\title[$C(K)$ spaces without norming M-bases]{Banach spaces of continuous functions\\without norming Markushevich bases}

\author[T.~Russo]{Tommaso Russo}
\address[T.~Russo]{Universit\"{a}t Innsbruck, Department of Mathematics, Technikerstra\ss e 13, 6020 Innsbruck, Austria; and Department of Mathematics, Faculty of Electrical Engineering, Czech Technical University in Prague, Technick\'a 2, 166 27 Prague 6, Czech Republic \newline
\href{https://orcid.org/0000-0003-3940-2771}{ORCID: \texttt{0000-0003-3940-2771}}}
\email{tommaso.russo@uibk.ac.at, tommaso.russo.math@gmail.com}

\author[J.~Somaglia]{Jacopo Somaglia}
\address[J.~Somaglia]{Politecnico di Milano, Dipartimento di Matematica, Piazza Leonardo da Vinci 32, 20133 Milano, Italy \newline
\href{https://orcid.org/0000-0003-0320-3025}{ORCID: \texttt{0000-0003-0320-3025}}}
\email{jacopo.somaglia@polimi.it}

\thanks{The research of the authors has been partially supported by GNAMPA (INdAM -- Istituto Nazionale di Alta Matematica).}

\date{\today}

\begin{abstract} We investigate the question whether a scattered compact topological space $K$ such that $C(K)$ has a norming Markushevich basis (M-basis, for short) must be Eberlein. This question originates from the recent solution, due to H\'ajek, Todor\v{c}evi\'c, and the authors, to an open problem from the Nineties, due to Godefroy. Our prime tool consists in proving that $C([0,\omega_1])$ does not embed in a Banach space with a norming M-basis, thereby generalising a result due to Alexandrov and Plichko. Subsequently, we give sufficient conditions on a compact $K$ for $C(K)$ not to embed in a Banach space with a norming M-basis. Examples of such conditions are that $K$ is a $0$-dimensional compact space with a P-point, or a compact tree of height at least $\omega_1 +1$. In particular, this allows us to answer the said question in the case when $K$ is a tree and to obtain a rather general result for Valdivia compacta. Finally, we give some structural results for scattered compact trees; in particular, we prove that scattered trees of height less than $\omega_2$ are Valdivia.
\end{abstract}
\maketitle

\section{Introduction}
A family of vectors $\{e_\alpha; \p_\alpha\}_{\alpha\in \Gamma}\subseteq X \times X^*$ is a \emph{Markushevich basis} (M-basis, for short) for a Banach space $X$ if $\langle \p_\alpha, e_\beta \rangle= \delta_{\alpha,\beta}$, $\Span\{e_\alpha\}_{\alpha\in \Gamma}$ is dense in $X$, and $\Span\{\p_\alpha\}_{\alpha\in \Gamma}$ is $w^*$-dense in $X^*$. One of the main research areas in non-separable Banach spaces, very active since the late Sixties, consists in detecting classes of non-separable Banach spaces that can be characterised by the existence of M-bases with suitable additional properties. This for example led to the theory of WLD and Plichko spaces, \cite[Chapter 6]{HMVZ}. Arguably, the strongest such a property is that of \emph{shrinking} M-basis, obtained by requiring $\Span\{\p_\alpha\}_{\alpha\in \Gamma}$ to be (norm) dense in $X^*$. A result that will be important for us is that a Banach space admits a shrinking M-basis if and only if it is simultaneously Asplund and WCG, \cite[Theorem 6.3]{HMVZ}.

A class of M-bases that is well studied since the Seventies \cite{AP, CastilloSalguero, Godefroy, Hajek, HRST, JohnZizler} is that of norming M-bases. An M-basis $\{e_\alpha; \p_\alpha\}_{\alpha\in\Gamma}$ is \emph{norming} if, for some $\lambda \in (0,1]$,
$$\lambda\|x\|\leq \sup\big\{|\langle\p,x\rangle|\colon \p\in {\rm span}\{\p_\alpha\}_{\alpha\in\Gamma},\, \|\p\|\leq 1 \big\} \qquad(x\in X),$$
namely, if ${\rm span}\{\p_\alpha\} _{\alpha\in\Gamma}$ is a norming subspace for $X$. Results such as the aforementioned characterisation of Banach spaces with a shrinking M-basis, constructions involving Jayne--Rogers selectors \cite{Fabian dual LUR, FG}, and others seemed to indicate a close connection between norming M-bases and WCG Banach spaces. This was crystallised in two problems, the first being whether every WCG Banach space has a norming M-basis (John and Zizler, \cite{JohnZizler}) and the second whether Asplund spaces with norming M-bases must be WCG (Godefroy, \cite{AP}, \cite[p.~211]{HMVZ}, \cite[Problem~112]{GMZ}). Nevertheless, both problems have been recently solved in the negative \cite{Hajek, HRST} (we also refer to the same papers for a more detailed history and motivation for the problems).

Our motivation in this paper comes from the construction in \cite{HRST}, where the counterexample is a subspace of $C(\mathcal{K}_\varrho)$, for a certain compact space $\mathcal{K}_\varrho$, which is scattered and not Eberlein (but it is not known if $C(\mathcal{K}_\varrho)$ admits a norming M-basis). This naturally suggests considering Godefroy's problem in the subclass of $C(K)$ spaces. Via well-known characterisations of WCG or Asplund $C(K)$ spaces (that we will recall in Section \ref{sec: prelim}), the problem can be restated as follows.
\begin{problem}[Godefroy]\label{pb: Gilles} Let $K$ be a scattered compact space such that $C(K)$ has a norming M-basis. Must $K$ be Eberlein?
\end{problem}

Even though we didn't manage to solve the problem in its full generality, we do obtain a rather general result for Valdivia compacta (Remark \ref{rmk: Kro and around}) and a complete solution for compact trees (Theorem \ref{th: sol for trees}). We should also mention that the problem was solved for adequate compacta in \cite[\S 5]{HRST}. When investigating the problem, one is naturally led to consider the Banach space $C([0,\omega_1])$. Indeed, in the simpler case when $C(K)$ has a $1$-norming M-basis, then $K$ is Valdivia, \cite[Theorem 5.3]{Kalenda survey} (here we are using that $K$ is scattered). Therefore, if $K$ is not Eberlein, it must contain a copy of $[0,\omega_1]$, \cite{Alster, DG}. For such a space, an important result due to Alexandrov and Plichko \cite{AP} is that $C([0,\omega_1])$ does not have a norming M-basis. The first main ingredient in our paper is the following more general statement.

\begin{mainth}\label{mth: AP} $C([0,\omega_1])$ embeds in no Banach space with a norming M-basis.
\end{mainth}
The proof of Theorem \ref{mth: AP} will be given in Section \ref{sec: AP subspace}. As it turns out, the proof is essentially based on the same argument as in \cite{AP}, more precisely we follow the alternative formulation given in \cite[Theorem 5.25]{HMVZ}. However, our more general argument requires some additional care, therefore we decided to present the argument in detail for convenience of the reader. In essence, the main idea consists in building a certain `ziqqurat like' function (see the definition of the function $z_N$ in \eqref{eq: ziqqurat}) and it originates with Alexandrov's paper \cite{Alexandrov}. The same type of function is then also used in \cite{AP}, \cite[Theorem 5.25]{HMVZ}, and \cite{Kalenda omega2}.\smallskip

Subsequently, in Section \ref{sec: consequences} we spell out explicit examples of Banach spaces which Theorem \ref{mth: AP} applies to and, in particular, we focus on $C(K)$ spaces. The canonical way to find a copy of $C([0,\omega_1])$ inside $C(K)$ for some $K$ is, of course, to find a continuous, surjective map from $K$ to $[0,\omega_1]$. This simple observation then leads us to our second main result.

\begin{mainth}\label{mth: C(K)} If $K$ is a compact topological space such that $[0,\omega_1]$ is a continuous image of $K$, then $C(K)$ does not embed in a Banach space with norming M-basis.\\
In particular, this applies to:
\begin{romanenumerate}
    \item\label{mth: ordinal} the ordinal intervals $[0,\eta]$, for every $\eta\geq \omega_1$;
    \item\label{mth: p-point} $0$-dimensional compacta that admit a P-point;
    \item\label{mth: trees} trees of height at least $\omega_1+1$ (endowed with the coarse wedge topology).
\end{romanenumerate}
\end{mainth}

As we said, the first part of the result follows from Theorem \ref{mth: AP} and standard facts, and it directly gives item (\ref{mth: ordinal}). In particular, we get that Banach spaces such as $C([0,\eta])$ for $\eta\geq \omega_1$, $C([0,\omega_1]\times [0,1])$, $C([0,\omega_1]^2)$, or $C([0,2\omega_1])$ do not admit norming M-bases. Recall that, by \cite{Semadeni}, $C([0,2\omega_1])= C([0,\omega_1]) \oplus C([0,\omega_1])$ is not isomorphic to $C([0,\omega_1])$, hence the non-existence of a norming M-basis in $C([0,2\omega_1])$ does not follow from \cite{AP}. Clause (\ref{mth: p-point}) instead is a consequence of a study of those compacta that admit continuous images onto $[0,\omega_1]$, which is the core part of Section \ref{sec: consequences} (Theorem \ref{th: scc and P-point}). Finally, (\ref{mth: trees}) is proved at the beginning of Section \ref{sec: trees}. In the same section we also give a positive answer to Problem \ref{pb: Gilles} when $K$ is a tree endowed with the coarse wedge topology (the definitions of trees and of the coarse wedge topology will be recalled at the beginning of that section). The subsequent, and main, part of Section \ref{sec: trees} is then dedicated to a study of scattered trees with the coarse wedge topology. In particular, the main result (Corollary \ref{cor: scattered tree is Valdivia}) asserts that scattered trees of height less than $\omega_2$ are Valdivia, which complements a result from \cite{S20}.

Having seen that trees of height at least $\omega_1+1$ can't admit norming M-bases, it is natural to ask what could happen in the case of height at most $\omega_1$. Under this assumption, it is known and not hard to infer that the tree $T$ must be Corson, \cite[Theorem 2.8]{N2}. Therefore, when $T$ is additionally scattered, it follows from \cite{Alster} that it must be Eberlein; hence $C(T)$ is simultaneously Asplund and WCG, which implies that it even admits a shrinking M-basis. Another case when the existence of a norming M-basis can be assured is if $T$ is metrisable, for in this case $C(T)$ is separable and the existence of a norming M-basis is then classical, \cite{Markushevich}.  However, let us mention that it seems open whether $C(T)$ admits a norming M-basis for all trees $T$ of height at most $\omega_1$ (Problem \ref{pb: norming in tree}).

\section{Preliminaries}\label{sec: prelim}
Throughout the paper, all topological spaces are assumed to be Hausdorff. For a topological space $K$, $K'$ denotes the set of accumulation points. The \emph{Cantor--Bendixson derivative} of $K$ is defined inductively as follows: $K^{(0)}\coloneqq K$, $K^{(\alpha+1)}\coloneqq K^{(\alpha)'}$, and if $\gamma$ is a limit ordinal $K^{(\gamma)}\coloneqq \bigcap_{\alpha< \gamma} K^{(\alpha)}$. A topological space $K$ is \emph{scattered} if every its closed subspace has an isolated point; in other words, $K$ contains no non-empty perfect subset. A topological space $K$ is scattered if and only if there is some ordinal $\alpha$ such that $K^{(\alpha)}= \emptyset$. In this case, such a minimal $\alpha$ is denoted by $\rk(K)$ and called the \emph{Cantor--Bendixson rank} of $K$. When $K$ is compact, a standard argument shows that $\rk(K)$ is always a successor ordinal. Every scattered compact is \emph{zero-dimensional}, \emph{i.e.}, it admits a basis for the topology consisting of clopen sets \cite[Theorem 29.7]{Willard}.

Given an arbitrary set $\Gamma$, the \emph{$\Sigma$-product} of real lines is
\begin{equation*}
\Sigma(\Gamma)=\{x\in [0,1]^{\Gamma}\colon|\{\gamma\in \Gamma\colon x(\gamma)\neq 0\}|\leq \omega\},
\end{equation*}
endowed with the restriction of the product topology from $[0,1]^{\Gamma}$. A compact space is \emph{Valdivia} if there exists a homeomorphic embedding $h\colon K \to [0,1]^{\Gamma}$ such that $h(K)\cap\Sigma(\Gamma)$ is dense in $h(K)$. In such a case, $\Sigma(K):=h^{-1}(\Sigma(\Gamma))$ is called a \emph{$\Sigma$-subset} of $K$. It is easy to verify that $\Sigma(K)$ is countably closed in $K$ and Fr\'{e}chet--Urysohn, \cite[Lemma 1.6]{Kalenda survey}. A compact space $K$ is \emph{Corson} if it is homeomorphic to a subset of $\Sigma(\Gamma)$, for some set $\Gamma$. A compact space is \emph{Eberlein} if it is homeomorphic to a subset of $c_0(\Gamma)$ endowed with the pointwise topology, for some set $\Gamma$. From these definitions it is clear that every Eberlein compact space is Corson and every Corson compactum is Valdivia. In general the converses fail to hold, however there are notable cases where the implications can be reversed. First, an important result due to Alster \cite{Alster} asserts that scattered Corson compacta are Eberlein. Moreover, Deville and Godefroy \cite{DG} proved that a Valdivia compact space is Corson if and only if it does not contain $[0,\omega_1]$.\smallskip

Next, we briefly recall the definitions of some classes of Banach spaces. A Banach space $X$ is \emph{Asplund} if every separable subspace of $X$ has a separable dual. Importantly, a $C(K)$ space is Asplund precisely when the compact space $K$ is scattered. A Banach space is \emph{weakly compactly generated} (WCG, for short) if it can be expressed as the closed linear span of some weakly compact subset. In the realm of $C(K)$ spaces, WCG spaces are characterised by $K$ being Eberlein, \cite{AmirLind}.

Finally, we recall some notions on biorthogonal systems in Banach spaces.  A \emph{biorthogonal system} in a Banach space $X$ is a system $\{e_\alpha; \p_\alpha\}_{\alpha\in\Gamma} \subseteq X \times X^*$ such that $\langle\p_\alpha, e_\beta \rangle=\delta_{\alpha,\beta}$ ($\alpha,\beta\in \Gamma$). A \emph{Markushevich basis} (henceforth, M-basis) is a biorthogonal system with the properties that $\Span\{e_\alpha\}_{\alpha\in \Gamma}$ is dense in $X$ and $\Span\{\p_\alpha\}_{\alpha\in \Gamma}$ is $w^*$-dense in $X^*$. The M-basis is shrinking if $\Span\{\p_\alpha\}_{\alpha\in \Gamma}$ is norm dense in $X^*$. It is \emph{$\lambda$-norming} ($0<\lambda \leq1$) if
$$\lambda\|x\|\leq \sup\big\{|\langle\p,x\rangle|\colon \p\in {\rm span}\{\p_\gamma\}_{\gamma\in\Gamma},\, \|\p\|\leq 1 \big\} \qquad(x\in X),$$
namely, if ${\rm span}\{\p_\gamma\} _{\gamma\in\Gamma}$ is a $\lambda$-norming subspace for $X$. $\{x_\gamma; \p_\gamma\}_{\gamma\in\Gamma}$ is \emph{norming} if it is $\lambda$-norming, for some $\lambda>0$. Among the many uses of M-bases, let us single out the following property that we shall use: every Banach space with an M-basis admits a continuous injection in $c_0(\Gamma)$. More precisely, if an M-basis satisfies $\|\p_\alpha\| =1$, then 
\begin{equation}\label{eq: coord in c0}
\left\{\langle \p_\alpha, x \rangle \colon \alpha\in \Gamma\right\}\in c_0(\Gamma).
\end{equation}
For additional information and reference on these notions, we shall refer the reader, \emph{e.g.}, to \cite{CCS, DGZ, FHHMZ, HMVZ, HRST, Kalenda survey} and the reference therein.

\section{Embeddings of \texorpdfstring{$C([0,\omega_1])$}{C([0,ω1])} and norming M-bases}\label{sec: AP subspace}
This section is dedicated to the proof of Theorem \ref{mth: AP}, whose statement is also recalled below, for convenience of the reader. Before we begin, let us explain one piece of notation that we shall use throughout the section. We identify the Banach space $C_0([0,\omega_1))$ of continuous functions that vanish at infinity on the locally compact space $[0,\omega_1)$ with the hyperplane of $C([0,\omega_1])$ given by functions that are zero at the point $\omega_1$. For a function $f\in C_0([0,\omega_1))$, we define the \emph{support} of $f$ to be $\supp(f)\coloneqq \overline{\{\gamma\in [0,\omega_1)\colon f(\gamma) \neq0\}}$. If $\mu$ is a functional on $C_0([0,\omega_1))$, we will identify it with the corresponding Radon measure $\mu\in M([0,\omega_1))$. Since $[0,\omega_1)$ is scattered, every such a measure is a weighted series of Dirac measures, \cite{Rudin}; hence $M([0,\omega_1))$ is isometric to $\ell_1([0,\omega_1))$. Therefore, we can define the \emph{support} of $\mu$ to be
\begin{equation*}
    \supp(\mu)\coloneqq \{\gamma\in [0,\omega_1)\colon \mu(\{ \gamma\}) \neq 0\}.
\end{equation*}
Let us alert the reader that this definition does not agree with the general definition of support of a measure, but it actually coincides with the support of a vector with respect to the canonical Schauder basis of $\ell_1([0,\omega_1))$.

\begin{theorem} Assume that a Banach space $X$ contains an isomorphic copy of $C([0,\omega_1])$. Then $X$ admits no norming M-basis.
\end{theorem}

\begin{proof} Towards a contradiction, let us assume that $X$ admits a norming M-basis and let us fix one, say $\{e_\alpha; \p_\alpha\}_{\alpha\in \Gamma}$. Fix a scalar $\lambda>0$ such that $\{e_\alpha; \p_\alpha\} _{\alpha\in \Gamma}$ is $\lambda$-norming.  Without loss of generality, we may assume that $\| \p_\alpha\|=1$ for every $\alpha\in \Gamma$. Now, our assumption implies in particular the existence of an isomorphic embedding $T\colon C_0([0,\omega_1))\to X$. Then, for every $\alpha\in \Gamma$, $T^*\p_\alpha$ is a functional on $C_0([0,\omega_1))$, hence it is represented by a measure $\mu_\alpha\in M([0,\omega_1))$. In other words, for every $f\in C_0([0,\omega_1))$ we have
$$\langle\mu_\alpha, f\rangle= \langle\p_\alpha, Tf\rangle.$$
Let us stress that it may very well happen that $\mu_\alpha= \mu_\beta$ for some distinct $\alpha,\beta\in \Gamma$, or that some $\mu_\alpha$ is equal to zero. We begin with some properties of the measures $\mu_\alpha$, $\alpha\in \Gamma$.

\begin{claim}\label{claim: nonzero mass} The set $\{\alpha\in \Gamma \colon \mu_\alpha ([0,\omega_1)) \neq 0\}$ is countable.
\end{claim}
\begin{proof}[Proof of Claim \ref{claim: nonzero mass}] \renewcommand\qedsymbol{$\square$} Assume, by contradiction, that there are uncountably many indices $\alpha\in \Gamma$ such that $\mu_\alpha ([0,\omega_1)) \neq 0$. Therefore, for some $\e >0$, there are infinitely many indices $(\alpha_j)_{j\in \omega}$ with $\left|\mu_{\alpha_j} ([0,\omega_1))\right| \geq \e$ for every $j\in \omega$. Every measure $\mu_{\alpha_j}$ is supported on a countable subset of $[0,\omega_1)$, so there exists $\gamma <\omega_1$ such that $\supp(\mu_{\alpha_j})\subseteq [0,\gamma]$ for all $j\in \omega$. Thus
$$\mu_{\alpha_j}([0,\omega_1))= \mu_{\alpha_j}([0,\gamma])= \langle \mu_{\alpha_j}, \bone_{[0,\gamma]} \rangle = \langle \p_{\alpha_j}, T\bone_{[0,\gamma]} \rangle,$$
where we are using that $\bone_{[0,\gamma]}\in C_0([0,\omega_1))$, as $[0,\gamma]$ is a clopen set. Hence, $\left| \langle \p_{\alpha_j}, T\bone_{[0,\gamma]} \rangle \right|\geq \e$ for every $j\in\omega$. However, this is in contradiction with $\left\{\langle \p_\alpha, T\bone_{[0,\gamma]} \rangle \colon \alpha\in \Gamma\right\}\in c_0(\Gamma)$, by \eqref{eq: coord in c0}, and proves the claim.
\end{proof}

\begin{claim}\label{claim: measures slide} For all $\gamma <\omega_1$ the set $\{\alpha\in \Gamma\colon \mu_\alpha(\{\gamma\}) \neq 0\}$ is countable.
\end{claim}
\begin{proof}[Proof of Claim \ref{claim: measures slide}] \renewcommand\qedsymbol{$\square$} For all $\gamma< \omega_1$, $C([0,\gamma])\coloneqq \{f\in C_0([0,\omega_1))\colon \supp(f)\subseteq [0,\gamma]\}$ is separable; hence, $T(C([0,\gamma]))$ is a separable subspace of $X$. Therefore, there is a countable subset $\Gamma_0$ of $\Gamma$ such that $T(C([0,\gamma])) \subseteq \overline{\rm span}\left\{e_\alpha \right\}_{\alpha\in \Gamma_0}$. Thus, if $\alpha\in \Gamma \setminus \Gamma_0$, we have
$$0= \langle \p_\alpha, Tf\rangle= \langle \mu_\alpha, f\rangle, \quad \text{ for all } f\in C([0,\gamma]).$$
So, $\mu_\alpha \cut_{C([0,\gamma])}=0$ for all $\alpha\in \Gamma \setminus \Gamma_0$, and we are done.
\end{proof}

As a consequence of Claim \ref{claim: measures slide} and the fact that every $\mu_\alpha$ is countably supported, we are now in position to perform the following standard closing-off argument.
\begin{fact}\label{fact: closing off} For every $\gamma<\omega_1$, there is $\tilde{\gamma}<\omega_1$ with $\gamma < \tilde{\gamma}$ and such that
$$\supp(\mu_\alpha)\cap [0,\tilde{\gamma})\neq \emptyset \quad \implies \quad \supp(\mu_\alpha) \subseteq [0,\tilde{\gamma}).$$
\end{fact}
\begin{proof}[Proof of Fact \ref{fact: closing off}] \renewcommand\qedsymbol{$\square$} We first show that there is $\gamma_1< \omega_1$ with $\gamma< \gamma_1$ such that $\supp(\mu_\alpha)\subseteq [0,\gamma_1]$, whenever $\supp(\mu_\alpha) \cap [0,\gamma] \neq \emptyset$. Indeed, according to Claim \ref{claim: measures slide}, the support of only countably many measures $\mu_\alpha$ intersects $[0,\gamma]$. Each of these measures is countably supported, thus the union of the supports is countable and we can take any $\gamma< \gamma_1<\omega_1$ which is larger than such a union. Repeating this argument by induction and setting $\gamma_0 \coloneqq \gamma$, we thus obtain a strictly increasing sequence $(\gamma_k)_{k \in\omega}$ with the property that $\supp(\mu_\alpha)\subseteq [0,\gamma_{k+1}]$, whenever $\supp(\mu_\alpha) \cap [0,\gamma_k] \neq \emptyset$ (for every $k\in \omega$). Finally, $\tilde{\gamma}\coloneqq \sup_{k\in \omega}\gamma_k <\omega_1$ is as desired. Indeed, if $\supp(\mu_\alpha)\cap [0,\tilde{\gamma})\neq \emptyset$, then for some $k\in\omega$ we have $\supp(\mu_\alpha)\cap [0,\gamma_k]\neq \emptyset$, so $\supp(\mu_\alpha)\subseteq [0,\gamma_{k+1}]\subseteq [0,\tilde{\gamma})$.
\end{proof}

Next, we apply Fact \ref{fact: closing off} inductively and build a sequence of consecutive intervals as follows. By Claim \ref{claim: nonzero mass} there are only countably many $\alpha\in\Gamma$ with $\mu_\alpha([0,\omega_1))\neq 0$, so we can choose $a_0<\omega_1$ such that $\supp(\mu_\alpha) \subseteq [0,a_0)$ for all such $\alpha$'s. Set $a_1\coloneqq \tilde{a_0}$ and, inductively, $a_{k+1}\coloneqq \tilde{a_k}$. Hence, we have an increasing sequence $(a_k)_{k\in\omega}\subseteq [0,\omega_1)$ with the property that
\begin{equation}\label{eq: ak support}
    \supp(\mu_\alpha)\cap [0,a_k)\neq \emptyset \quad \implies \quad \supp(\mu_\alpha) \subseteq [0,a_k), \quad \mbox{for all }k\geq 1.    
\end{equation}
We now define intervals $I_k\coloneqq [a_k,a_{k+1})$, for all $k\geq 1$ and notice that \eqref{eq: ak support} readily gives
\begin{equation}\label{eq: Ik support}\tag{$\dagger$}
    \supp(\mu_\alpha)\cap I_k\neq \emptyset \quad \implies \quad \supp(\mu_\alpha) \subseteq I_k, \quad \mbox{for all }k\geq 1.    
\end{equation}
Indeed, $\supp(\mu_\alpha)\cap I_k\neq \emptyset$ and \eqref{eq: ak support} imply $\supp(\mu_\alpha)\subseteq [0,a_{k+1})$; likewise, $\supp(\mu_\alpha)\cap [0,a_k)\neq \emptyset$ would imply $\supp(\mu_\alpha) \subseteq [0,a_k)$, so $\supp(\mu_\alpha)\cap I_k= \emptyset$, a contradiction. Hence, $\supp(\mu_\alpha)\cap [0,a_k)= \emptyset$ and $\supp(\mu_\alpha) \subseteq I_k$. \smallskip

After this preparation, we are ready for the main `staircase' construction. Fix $N\in\N$ and define
\begin{equation}\label{eq: ziqqurat}
    z_N\coloneqq \sum_{k=1}^N \frac{k}{N} \bone_{(a_k,a_{k+1}]} + \sum_{k=N+1}^{2N-1} \frac{2N-k}{N} \bone_{(a_k,a_{k+1}]}.
\end{equation}
Then, $z_N\in C_0([0,\omega_1))$ (since the sets $(a_k,a_{k+1}]$ are clopen), $\|z_N\|=1$, and $\supp(z_N)=(a_1,a_{2N}]\subseteq I_1\cup\dots \cup I_{2N}$. We shall prove that, for $N$ large enough, $Tz_N\in X$ witnesses that the M-basis $\{e_\alpha; \p_\alpha\}_{\alpha\in \Gamma}$ is not $\lambda$-norming. Therefore, fix any $\p\in \Span \{ \p_\alpha\}_{\alpha\in \Gamma}$ with $\|\p\| \leq 1$; we want to estimate $\langle\p, Tz_N \rangle = \langle T^*\p, z_N\rangle$. By definition, we have $T^*\p\in \Span \{\mu_\alpha \}_{\alpha\in\Gamma}$, so let us write $T^*\p = \sum_{\alpha\in F}c_\alpha \mu_\alpha$, where $F\subseteq \Gamma$ is a finite set. In the light of \eqref{eq: Ik support}, for every $\alpha\in F$ we have two cases: either $\supp(\mu_\alpha)$ is contained in $I_k$ for some $k=1,\dots,2N$, or $\supp(\mu_\alpha)$ is disjoint from $I_1\cup\dots \cup I_{2N}$. Therefore, we may write $T^*\p = h + \sum_{k=1}^{2N} h_k$, where $\supp(h_k)\subseteq I_k$ and $\supp(h)\cap (I_1\cup \dots \cup I_{2N})= \emptyset$. Consequently, $\langle h, z_N\rangle=0$ and $h_k(I_k)=0$, due to Claim \ref{claim: nonzero mass} and our choice of $a_0$.

With this information at our disposal, we fix $k=1, \dots, N$ and compute
\begin{eqnarray*}
    \langle h_k, z_N\rangle &=& \frac{k-1}{N} h_k((a_{k-1}, a_k]) + \frac{k}{N} h_k((a_k, a_{k+1}])\\
    &=& - \frac{1}{N} h_k((a_{k-1}, a_k]) + \frac{k}{N} h_k((a_{k-1}, a_{k+1}])\\
    &=& - \frac{1}{N} h_k(\{a_k\}) + \frac{k}{N} h_k(I_k)\\
    &=& - \frac{1}{N} h_k(\{a_k\}) = - \frac{1}{N} T^*\p (\{a_k\}).
\end{eqnarray*}
Similarly, for $k=N+1, \dots, 2N$, we have
\begin{eqnarray*}
    \langle h_k, z_N\rangle &=& \frac{2N+1-k}{N} h_k((a_{k-1}, a_k]) + \frac{2N-k}{N} h_k((a_k, a_{k+1}])\\
    &=& \frac{1}{N} h_k((a_{k-1}, a_k]) + \frac{2N -k}{N} h_k((a_{k-1}, a_{k+1}])\\
    &=& \frac{1}{N} h_k(\{a_k\}) + \frac{2N -k}{N} h_k(I_k) = \frac{1}{N} T^*\p (\{a_k\}).
\end{eqnarray*}
Consequently, when we put the two cases together, we obtain
\begin{equation*}
    |\langle \p, Tz_N\rangle| = |\langle T^*\p, z_N\rangle| \leq \frac{1}{N} \sum_{k=1}^{2N} \left| T^*\p (\{a_k\})\right| \leq \frac{1}{N} \|T^*\p\| \leq \frac{1}{N} \|T\|,
\end{equation*}
for every $\p\in \Span\{\p_\alpha\}_{\alpha\in\Gamma}$ with $\|\p\|\leq 1$. On the other hand, we have
\begin{equation*}
    \|T^{-1}\|^{-1}\leq \|Tz_N\|\leq \frac{1}{\lambda} \sup\left\{ |\langle \p, Tz_N\rangle|\colon \p\in \Span\{\p_\alpha\}_{\alpha\in\Gamma}, \|\p\|\leq 1 \right\} \leq \frac{1}{\lambda N} \|T\|,
\end{equation*}
which, for large enough $N$, gives us the desired contradiction and concludes the proof.
\end{proof}

\section{Chains of clopen sets, P-points, and applications to \texorpdfstring{$C(K)$}{C(K)} spaces}\label{sec: consequences}
In this section we obtain some consequences of Theorem \ref{mth: AP} and in particular we give explicit sufficient conditions on $K$ to guarantee that $C(K)$ does not embed in a space with a norming M-basis. In particular, we prove clauses (\ref{mth: ordinal}) and (\ref{mth: p-point}) of Theorem \ref{mth: C(K)}. In the last part of the section, we discuss a possible approach to Problem \ref{pb: Gilles} using these results.

We begin with the following simple consequence of Theorem \ref{mth: AP}, that we already mentioned in the Introduction.
\begin{proposition}\label{prop: cont image} Let $K$ be compact and such that $[0,\omega_1]$ is a continuous image of $K$. Then $C(K)$ does not embed in a Banach space with a norming M-basis.
\end{proposition}
\begin{proof} Take a continuous surjection $\p\colon K \to [0,\omega_1]$. Then the map $f\mapsto f\circ \p$ defines an isometric embedding of $C([0,\omega_1])$ in $C(K)$, so Theorem \ref{mth: AP} concludes the proof.
\end{proof}

We now enter the main part of the section, where we characterise those compact spaces that admit a continuous image onto $[0,\omega_1]$. The crucial notion will be that of a chain of clopen sets, that we now introduce.
\begin{definition} Let $K$ be a compact topological space. A \emph{strict $\lambda$-chain of clopen sets} (\emph{$\lambda$-scc}, for short) is a collection $(C_\alpha)_{\alpha<\lambda}$ of nonempty clopen subsets of $K$ such that
$$C_\alpha \subsetneq C_\beta\qquad \text{ if } \alpha<\beta.$$
A $\lambda$-scc is \emph{thin} if $\overline{\bigcup_{\alpha<\lambda}C_\alpha}$ is a clopen set in $K$ and
$$\overline{\bigcup_{\alpha<\lambda}C_\alpha} \setminus \bigcup_{\alpha<\lambda}C_\alpha \text{ is a singleton}.$$
\end{definition}

Notice that if $\beta< \lambda$ is a limit ordinal, then $\bigcup_{\alpha < \beta} C_\alpha\neq C_\beta$, because alternatively $\{C_\alpha\}_{\alpha< \beta}$ would be an open cover of the compact $C_\beta$, with no finite subcover. By the same reason, if $\lambda$ is limit $\bigcup_{\alpha<\lambda}C_\alpha$ cannot be a closed set; therefore the set
$$\overline{\bigcup_{\alpha<\lambda}C_\alpha} \setminus \bigcup_{\alpha<\lambda}C_\alpha$$
is never empty. We will see the distinction between the two notions in the proof of Theorem \ref{th: scc and P-point} and in Example \ref{ex: scc} below.

Before our result, we need to recall another topological notion. A point $p$ in a topological space $K$ is a \emph{P-point} if $p$ is not isolated and, for every countable family $(U_j)_{j\in \omega}$ of neighbourhoods of $p$, $\bigcap_{j\in \omega} U_j$ is a neighbourhood of $p$ as well. The archetypal example of a P-point is $\omega_1$ in $[0,\omega_1]$.

\begin{theorem}\label{th: scc and P-point} Let $K$ be compact and consider the following conditions:
\begin{romanenumerate}
    \item\label{item: cont image} $[0,\omega_1]$ is a continuous image of $K$;
    \item\label{item: scc} $K$ admits an $\omega_1$-scc;
    \item\label{item: thin scc} $K$ admits a thin $\omega_1$-scc;
    \item\label{item: P-point} $K$ has a P-point.
\end{romanenumerate}
Then {\normalfont \eqref{item: thin scc}}$\implies${\normalfont \eqref{item: scc}}$\iff${\normalfont \eqref{item: cont image}} and {\normalfont \eqref{item: thin scc}}$\implies${\normalfont \eqref{item: P-point}}. If $K$ is $0$-dimensional, then {\normalfont \eqref{item: P-point}}$\implies${\normalfont \eqref{item: scc}}.\\
Finally, if $K$ is scattered and Valdivia, then {\normalfont \eqref{item: scc}}$\implies${\normalfont \eqref{item: thin scc}}, hence all conditions are equivalent.
\end{theorem}

It is natural to ask whether the equivalence in the theorem can be used to give a positive answer to Problem \ref{pb: Gilles}, at least when $K$ is Valdivia. We will see in Remark \ref{rmk: Kro and around} below that this is not the case, but in a sense we do cover most scattered Valdivia compacta.

\begin{proof} That \eqref{item: thin scc}$\implies$\eqref{item: scc} is obvious.

\eqref{item: cont image}$\implies$\eqref{item: scc}: Let $\p\colon K\to[0,\omega_1]$ be a continuous surjection and define $C_\alpha\coloneqq \p^{-1}([0,\alpha])$, for all $\alpha< \omega_1$. Then $(C_\alpha)_{\alpha<\omega_1}$ is an $\omega_1$-scc.

Conversely, for \eqref{item: scc}$\implies$\eqref{item: cont image}, take an $\omega_1$-scc $(C_\alpha)_{\alpha<\omega_1}$ and define a function $\p\colon K \to [0,\omega_1]$ as follows:
\begin{equation*}
    \p(x)\coloneqq \begin{cases} \min\{\alpha<\omega_1 \colon x\in C_\alpha\}& \mbox{if } x\in \bigcup_{\alpha<\omega_1} C_\alpha,\\
    \omega_1 & \mbox{elsewhere.}
    \end{cases}
\end{equation*}
Then, by definition, $C_\alpha=\{x\in K\colon \p(x)\leq \alpha\}$. Therefore, when $\delta< \alpha< \omega_1$, we have that $\p^{-1}((\delta, \alpha])= C_\alpha \setminus C_\delta$ is a clopen set; similarly, for $\delta< \omega_1$, $\p^{-1}((\delta, \omega_1])= K \setminus C_\delta$ is also clopen. Thus $\p$ is continuous. Finally, since $\p^{-1}(\{0\})= C_0\neq \emptyset$ and $\p^{-1}(\{\alpha+1\})= C_{\alpha+1}\setminus C_\alpha\neq \emptyset$ we conclude that the image of $\p$ is dense in $[0,\omega_1]$. By compactness of $K$, $\p$ must then be surjective, which proves the implication.

\eqref{item: thin scc}$\implies$\eqref{item: P-point}: Let $(C_\alpha)_{\alpha<\omega_1}$ be a thin $\omega_1$-scc, define the clopen set $C_{\omega_1}\coloneqq \overline{\bigcup_{\alpha< \omega_1} C_\alpha}$, and let $p\in K$ be the unique point such that $\{p\}= C_{\omega_1} \setminus{\bigcup_{\alpha< \omega_1} C_\alpha}$. We claim that $p$ is a P-point. First, $p$ is not an isolated point, since, by definition, $p\in \overline{\bigcup_{\alpha< \omega_1} C_\alpha} \setminus \bigcup_{\alpha< \omega_1} C_\alpha$. Now, for $\alpha<\omega_1$ consider the clopen set $D_\alpha \coloneqq C_{\omega_1}\setminus C_\alpha$; notice that, by definition, $\bigcap_{\alpha<\omega_1}D_\alpha=\{p\}$. Therefore, if $(U_j)_{j\in \omega}$ is any sequence of neighbourhoods of $p$, then we have $\bigcap_{\alpha< \omega_1} D_\alpha \subseteq U_j$, for any $j\in\omega$. By a standard compactness argument, we deduce that there exists $\alpha_j< \omega_1$ such that $D_{\alpha_j} \subseteq U_j$ (the sets $(D_\alpha \setminus U_j) _{\alpha< \omega_1}$ are a decreasing family of closed sets with empty intersection). Finally, letting $\alpha\coloneqq \sup _{j\in\omega} \alpha_j$, we conclude that $D_\alpha \subseteq \bigcap_{j\in\omega} U_j$, whence $\bigcap_{j\in \omega} U_j$ is a neighbourhood of $p$. Thus $p$ is a P-point.

Next, we assume that $K$ is $0$-dimensional and we shall prove \eqref{item: P-point}$\implies $\eqref{item: scc}. Let $p\in K$ be a P-point; hence, every neighbourhood of $p$ is an infinite set. We will build a chain $(D_\alpha)_{\alpha< \omega_1}$ of clopen neighbourhoods of $p$ such that
\begin{equation*}
    D_\beta \subsetneq D_\alpha \quad \mbox{if } \alpha< \beta< \omega_1. 
\end{equation*}
Indeed, start with any clopen set $D_0\neq K$ that contains $p$. Having constructed $(D_\alpha)_{\alpha< \gamma}$, for some $\gamma< \omega_1$, we distinguish two cases. If $\gamma= \gamma' +1$, then we take a clopen neighbourhood $D_\gamma \subsetneq D_{\gamma'}$ of $p$. If $\gamma$ is a limit, by our assumption $\bigcap_{\alpha< \gamma} D_\alpha$ is a neighbourhood of $p$; hence, we can take a clopen set $D_\gamma \subsetneq \bigcap_{\alpha< \gamma} D_\alpha$ containing $p$. Finally, from the chain $(D_\alpha)_{\alpha< \omega_1}$, we define $C_\alpha\coloneqq K \setminus D_\alpha$ and $(C_\alpha)_{\alpha< \omega_1}$ is the desired $\omega_1$-scc.

Finally it remains to show \eqref{item: scc}$\implies$\eqref{item: thin scc}, whenever $K$ is scattered and Valdivia. Let $K$ be a scattered Valdivia compact space and $\Sigma(K)$ be a $\Sigma$-subset of $K$. Take an $\omega_1$-scc $(C_{\alpha})_{\alpha <\omega_1}$ in $K$. 
\begin{claim}\label{claim: the closure is clopen} $C_{\omega_1} \coloneqq\overline{\bigcup_{\alpha< \omega_1} C_{\alpha}}$ is a clopen subset of $K$.
\end{claim}
\begin{proof}[Proof of Claim \ref{claim: the closure is clopen}] \renewcommand\qedsymbol{$\square$} For each $\alpha<\omega_1$, we define $E_\alpha \coloneqq C_\alpha\cap \Sigma(K)$; notice that $\overline{E_\alpha}= C_\alpha$. The set $\bigcup_{\alpha<\omega_1}E_{\alpha}$ is, by definition, open in $\Sigma(K)$. Let us prove that $\bigcup_{\alpha<\omega_1}E_{\alpha}$ is also closed in $\Sigma(K)$. Indeed, since $\Sigma(K)$ is Fr\'{e}chet--Urysohn, if $t\in \overline{\bigcup _{\alpha< \omega_1} E_\alpha}$ (where the closure is taken in $\Sigma(K)$), there exists a sequence $(t_n)_{n\in\omega}\subseteq \bigcup_{\alpha<\omega_1}E_\alpha$ that converges to $t$. Let $\alpha_n<\omega_1$ be such that $t_n\in E_{\alpha_n}$. Letting $\alpha=\sup_{n\in\omega}\alpha_n$, we get $(t_n)_{n\in\omega}\subseteq E_\alpha$, thus $t\in E_\alpha$. Therefore, $\bigcup_{\alpha< \omega_1} E_\alpha$ is a clopen subset of $\Sigma(K)$. Next, we need to recall that $K$ is the \v{C}ech-Stone compactification of $\Sigma(K)$ (see \emph{e.g.}, \cite[Proposition 1.9]{Kalenda survey}), whence it is standard to deduce that $\overline{\bigcup _{\alpha< \omega_1} E_\alpha}$ is a clopen subset of $K$ (\cite[Corollary 3.6.5]{Eng}). Finally observing that  $\overline{\bigcup_{\alpha<\omega_1}E_{\alpha}}=\overline{\bigcup_{\alpha<\omega_1}C_{\alpha}}$ (recall that $C_{\alpha}=\overline{E_\alpha}$), we obtain the claim.
\end{proof}

As we already mentioned, $\bigcup_{\alpha< \omega_1} C_\alpha$ is not a closed set, so $C_{\omega_1}\setminus \bigcup_{\alpha< \omega_1} C_\alpha\neq \emptyset$. Since $K$ is scattered, we can pick an isolated point $p\in C_{\omega_1}\setminus \bigcup_{\alpha< \omega_1} C_\alpha $. In particular, there exists a clopen subset $U\subseteq C_{\omega_1}$ of $K$ such that $U\setminus \bigcup_{\alpha< \omega_1} C_\alpha=\{p\}$. In other words, we have $U\setminus \{p\}\subseteq \bigcup_{\alpha< \omega_1} C_\alpha$. We now consider the clopen sets $D_\alpha \coloneqq C_\alpha\cap U$ and $D_{\omega_1} \coloneqq U$. By definition we have $D_{\omega_1}\setminus \{p\}=\bigcup_{\alpha< \omega_1} D_\alpha$, which implies $\overline{\bigcup_{\alpha< \omega_1} D_\alpha}= D_{\omega_1}$ (because $p\in \overline{\bigcup_{\alpha< \omega_1} D_\alpha}$). Moreover, $(D_\alpha)_{\alpha< \omega_1}$ is clearly a non-decreasing family of clopen subsets of $K$. Finally, since $\bigcup_{\alpha< \omega_1} D_\alpha$ is not a closed set, the chain $(D_\alpha)_{\alpha< \omega_1}$ cannot stabilise. Thus, up to passing to an uncountable subfamily, we can additionally assume that $(D_\alpha)_{\alpha< \omega_1}$ constitutes an $\omega_1$-scc. The fact that such a chain is also thin has been proved above, so \eqref{item: thin scc} holds and the proof is concluded.
\end{proof}

Next, we shall give some examples to show that the implications in the previous theorem are best possible without the additional assumptions on $K$.
\begin{example}\label{ex: scc}
\begin{enumerate}
    
    \item[(a)] The long line (see, \emph{e.g.}, \cite[Theorem 5.2]{kubis retractions}) is a connected (Valdivia) compact set with a P-point. Hence (\ref{item: P-point})$\centernot \implies$(\ref{item: thin scc}), thus also (\ref{item: P-point})$\centernot \implies$(\ref{item: scc}) for spaces that are not $0$-dimensional.
    \item[(b)] The remainder $\omega^*\coloneqq \beta\omega \setminus \omega$ of the \v{C}ech--Stone compactification of $\omega$ admits a continuous image onto $[0,\omega_1]$, by Parovi\v{c}enko's theorem \cite{Parovicenko}. Yet, it is consistent that $\omega^*$ admits no P-point, \cite{Wimmers}. Aside for giving another (consistent) example that (\ref{item: scc})$\centernot \implies$(\ref{item: P-point}), it is a typical example of space admitting $[0,\omega_1]$ as continuous image, without containing copies of it.
    \item[(c)] Consider the compact topological space $K\coloneqq [0,\omega_1]\cup\{x_n\}_{n\in \omega}$, where $(x_n)$ is a sequence that converges to $\omega_1$. Then $C_\alpha\coloneqq [0,\alpha]$ ($\alpha< \omega_1$) defines an $\omega_1$-scc, but $K$ has no P-point, hence it also does not have a thin $\omega_1$-scc. Notice that such a $K$ is scattered, but not Valdivia, \cite[Example 1.10(iii)]{Kalenda survey}. So (\ref{item: scc}) implies neither (\ref{item: thin scc}) nor (\ref{item: P-point}), even for scattered compacta.
    \item[(d)] Finally, we give an example that (\ref{item: scc})$\centernot \implies$(\ref{item: thin scc}), for Valdivia, $0$-dimensional (non-scattered) compacta. Consider $K=[0,\omega_1]\times 2^{\omega}$; it is Valdivia and 0-dimensional, since both $[0,\omega_1]$ and $2^{\omega}$ are. An example of $\omega_1$-scc in $K$ is given by $C_\alpha \coloneqq [0,\alpha] \times 2^\omega$ ($\alpha< \omega_1$), thus $K$ satisfies \eqref{item: scc}. On the other hand, since $\{\alpha\}\times 2^{\omega}$ is metrizable for every $\alpha\leq\omega_1$, $K$ does not contain P-points. In particular, $K$ does not satisfy (\ref{item: thin scc}).
\end{enumerate}
\end{example}

\begin{remark} Example \ref{ex: scc}(c) and the results of this section imply that $C(\omega^*)= \ell_\infty / c_0$ does not embed in a Banach space with a norming M-basis. However, this is not a new result, since $\ell_\infty / c_0$ has no rotund norm \cite{Bourgain}, hence it does not embed in a Banach space with an M-basis at all. By the same reason, $\ell_\infty$ does not embed in a space with norming M-basis, as it admits no LUR norm, \cite[Theorem II.7.10]{DGZ} (while it embeds in a Banach space with an M-basis \cite[Theorem 4.63]{FHHMZ}). Notice however that $[0,\omega_1]$ is not continuous image of $\beta \omega$ (the latter being separable), hence our results give no information about $\ell_\infty$.
\end{remark}

\begin{remark}\label{rmk: Kro and around} One can be tempted to try applying Theorem \ref{th: scc and P-point} in combination with Proposition \ref{prop: cont image} to solve Problem \ref{pb: Gilles} at least for Valdivia compacta. Indeed, by the result of Alster that we mentioned already, it is enough to prove that $K$ is Corson. However, since $K$ is Valdivia, it either contains $[0,\omega_1]$, or it is Corson, by \cite{DG}. Thus, we can assume that $K$ contains $[0,\omega_1]$ and it is natural to try using this additional information to find an $\omega_1$-scc, or directly a P-point (indeed, then Theorem \ref{th: scc and P-point} and Proposition \ref{prop: cont image} would give the contradiction that $C(K)$ has no norming M-basis). However, this is (perhaps surprisingly) not possible in general: in fact the compact $\mathcal{K}_\varrho$ constructed in \cite{HRST} is a scattered Valdivia (even semi-Eberlein) compactum with no P-point. On the other hand, scattered Valdivia compacta with no weak P-points are Corson (see Lemma \ref{lem: no weak P-point} below); therefore, the unique scattered Valdivia compacta for which we can't answer Problem \ref{pb: Gilles} are those admitting some weak P-point, but no P-points (the compact $\mathcal{K}_\varrho$ from \cite{HRST} is such an example).
\end{remark}

Before we state and prove the lemma, we recall that a point $p$ in a topological space $K$ is a \textit{weak P-point} if $p$ is not isolated and it is accumulation point of no countable set in $K\setminus \{p\}$.

\begin{lemma}\label{lem: no weak P-point} Let $K$ be a scattered Valdivia compactum. If $K$ admits no weak P-point, then $K$ is Corson.
\end{lemma}
\begin{proof} Let $\Sigma(K)$ be a dense $\Sigma$-subspace of $K$. If $K$ is not Corson, then $K\setminus \Sigma(K)$ is not empty, so we can take a point $p$ that is isolated in $K\setminus \Sigma(K)$. Thus, there exists a neighbourhood $U$ of $p$ such that $U\setminus \{p\}\subseteq \Sigma(K)$. But $p$ is not isolated in $K$ ($\Sigma(K)$ is dense), hence, in order not to be a weak P-point, it must be an accumulation point of a countable set in $K$. Therefore, it is also an accumulation point of a countable subset of $U\setminus \{p\}\subseteq \Sigma(K)$. However, this contradicts that $\Sigma(K)$ is countably closed.
\end{proof}

To conclude the section, let us point out explicitly one consequence of Remark \ref{rmk: Kro and around}. Let $\mathscr{C}$ be a class of Valdivia compacta such that $K\in \mathscr{C}$ is Corson if and only if $K$ admits no P-point (for example, a class of Valdivia compacta in which every weak P-point is a P-point). Then Problem \ref{pb: Gilles} has a positive answer for all $K\in \mathscr{C}$. We don't have any specific class of compacta to which this comment applies (however, we will need a very similar argument in Theorem \ref{th: sol for trees}), but we decided to spell it out here for possible future reference.

\section{Compact trees and continuous functions thereon}\label{sec: trees}
In this section we specialise results from Section \ref{sec: consequences} and our object of interest is now the case where $K$ is a compact tree endowed with the coarse wedge topology. After a brief review of the necessary definitions, we show how to exploit the results from Section \ref{sec: consequences} to solve Problem \ref{pb: Gilles} in the case where $K$ is a tree. Next, we study in more detail the structure of scattered trees. We prove the almost optimal result that every scattered tree of height less than $\omega_2$ is Valdivia and we also give an order-theoretic characterisation of scattered trees. Finally, by means of an analogous characterisation of metrisable trees, we build an explicit example of a $1$-norming M-basis in $C(K)$, where $K$ is a metrisable compact tree. 

A \textit{tree} is a partially ordered set $(T,\leq)$ such that the set of predecessors $\{s\in T\colon s< t\}$ of any $t\in T$ is well-ordered by $<$. A tree $T$ is said to be \textit{rooted} if it has only one minimal element (denoted by $0_T$), called \textit{root}. For a tree $T$ and $t\in T$, we write $\hat{t}=\{s\in T\colon s\leq t\}$. A \textit{chain} in $T$ is a totally ordered subset of $T$; $T$ is called \textit{chain complete} if every chain has a supremum. For an element $t\in T$, $\htte(t,T)$ denotes the order type of $\{s\in T\colon s< t\}$. Given an ordinal $\alpha$, the set $\Lev_{\alpha}(T)=\{t\in T\colon \htte(t,T)=\alpha\}$ is called the \textit{$\alpha$-th level} of $T$. The \textit{height} of $T$, denoted by $\htte(T)$, is the least $\alpha$ such that $\Lev_{\alpha}(T)=\emptyset$. For an element $t\in T$, $\cf(t)$ denotes the cofinality of $\htte(t,T)$, where $\cf(t)=0$ if $\htte(t,T)$ is a successor ordinal or $\htte(t,T)=0$. Moreover, $\ims(t)=\{s\in T\colon t\leq s, \,\htte(s,T)=\htte(t,T)+1\}$ denotes the set of \textit{immediate successors} of $t$. We denote by $I(T)=\{t\in T\colon  \cf(t)<\omega\}$. A tree $T$ is an \textit{$r$-tree} if it satisfies the following condition
\[
|\ims(t)|<\omega \mbox{ whenever } \cf(t)\geq\omega_1.
\]

For $t\in T$ we put $V_t=\{s\in T\colon s\geq t\}$. The \textit{coarse wedge topology} on a tree $T$ is the one whose subbase is given by the sets $V_t$ and their complements, where $t\in I(T)$. Thus, a basis for the topology is formed by sets of the form $W_{t}^{F}=V_t\setminus \bigcup\{V_s\colon s\in F\},$
where $t\in I(T)$ and $F\subseteq \ims(t)$ is finite. Recall that a tree $T$ is compact Hausdorff in the coarse wedge topology if and only if $T$ is chain complete and it has finitely many minimal elements, \cite[Corollary 3.5]{N1}. For this reason, from now on we will only consider chain complete, rooted trees. Importantly, this assumption assures that every two elements $s,t\in T$ admit an infimum $s\wedge t$; indeed, $s\wedge t=\max(\hat{s}\cap\hat{t})$. We refer the reader to \cite{K20,N1,N2,RS,S18,S20} for more information on the coarse wedge topology and relation between this topology and classes of non-metrisable compacta. Let us only point out that by \cite[Theorem 2.8]{N2} a tree is Corson if and only if every chain is countable, in particular if $\htte(T)\leq\omega_1$, then $T$ is a Corson compact space. 

We begin by showing how to use the results from Section \ref{sec: consequences} to solve Problem \ref{pb: Gilles} for trees.
\begin{lemma}\label{lem: omega1 image} Let $T$ be a compact tree with $\htte(T)>\omega_1$. Then $[0,\omega_1]$ is a continuous order preserving image of $T$.
\end{lemma}
\begin{proof}It is enough to take any $t\in \Lev_{\omega_1}(T)$, consider the continuous function $f\colon T\to \hat{t}$ defined by $f(s)\coloneqq s \wedge t$, and recall that, by definition, $\hat{t}$ is homeomorphic to $[0,\omega_1]$.
\end{proof}

Combining the above lemma with Proposition \ref{prop: cont image}, we immediately arrive at the following corollary. 
\begin{corollary}\label{cor: no norming in tall trees} If $T$ is a compact tree with $\htte(T)>\omega_1$, then $C(T)$ does not embed in a space with norming M-basis.
\end{corollary}

In turn, this corollary allows us to answer Problem \ref{pb: Gilles} when $K$ is a tree.
\begin{theorem}\label{th: sol for trees} Let $T$ be a scattered compact tree such that $C(T)$ has a norming M-basis. Then $T$ is Eberlein.
\end{theorem}
\begin{proof} Our assumptions and Corollary \ref{cor: no norming in tall trees} imply that $\htte(T)\leq \omega_1$, from which it is elementary to conclude that $T$ is Corson, \cite[Theorem 2.8]{N2}. Finally, Alster's result that scattered Corson compacta are Eberlein \cite{Alster} concludes the proof.
\end{proof}

Notice that one step of the previous proof consisted in deriving in particular that $T$ must be Valdivia whenever $T$ is scattered and $C(T)$ has a norming M-basis. We now proceed to prove (in Corollary \ref{cor: scattered tree is Valdivia} below) a more general result without the assumption on the norming M-basis, under an additional assumption on the height of $T$.

Following \cite{K20, S20}, we say that a subset $A$ of a topological space $X$ is \textit{$\omega_1$-relatively discrete} if it can be written as a union of $\omega_1$-many  discrete subsets of $X$. A simple argument shows that $B$ is $\omega_1$-relatively discrete whenever $B\subseteq A$ and $A$ is $\omega_1$-relatively discrete. Moreover, if sets $(A_\alpha)_{\alpha< \omega_1}$ are $\omega_1$-relatively discrete, then the same holds for $\bigcup_{\alpha< \omega_1} A_\alpha$. As a particular case, if $A$ is $\omega_1$-relatively discrete and $|B|\leq\omega_1$, then $A\cup B$ is $\omega_1$-relatively discrete. For possibly larger unions we have the following fact.

\begin{fact}\label{fact: unionofomega1reldiscrete} Let $X$ be a topological space and $\{A_\alpha\}_{\alpha\in \Gamma}$ be a family of pairwise disjoint $\omega_1$-relatively discrete open subsets, for some set $\Gamma$. Then $A=\bigcup_{\alpha\in\Gamma}A_\alpha$ is $\omega_1$-relatively discrete. 
\end{fact}

\begin{proof} For every $\alpha\in \Gamma$ there exists a family $\{A_{\alpha,\xi}\}_{\xi<\omega_1}$ of discrete subsets of $A_\alpha$ such that $A_\alpha= \bigcup_{\xi< \omega_1} A_{\alpha,\xi}$. We define $B_{\xi}\coloneqq \bigcup_{\alpha\in \Gamma}A_{\alpha,\xi}$. Since the $A_\alpha$'s are pairwise disjoint and open, $A_{\alpha,\xi}= A_\alpha\cap B_\xi$ is open in $B_\xi$. Therefore, $B_\xi$ is discrete and we get that $A$ is $\omega_1$-relatively discrete.
\end{proof}

The proof of Corollary \ref{cor: scattered tree is Valdivia} will be a consequence of Theorem \ref{t: omega1discrete}, asserting that scattered trees of height at most $\omega_2$ are $\omega_1$-relatively discrete. The following lemma distils the crucial ingredient that we need.

\begin{lemma}\label{lem: S and discrete lev} Let $T$ be a scattered compact tree such that $\htte(T)\leq \omega_2$. Suppose that $T$ satisfies the following condition:
\begin{center} for every closed subset $S$ of $T$ that is $\wedge$-closed and such that $\rk(S)<\rk(T)$,\\ $S$ is $\omega_1$-relatively discrete in the subspace topology.
\end{center}
Then $T$ is $\omega_1$-relatively discrete.
\end{lemma}

\begin{proof} The rough idea of the argument is to split the tree $T$ into many clopen, $\wedge$-closed pieces with rank smaller than $\rk(T)$ and a remainder of cardinality at most $\omega_1$. Then Fact \ref{fact: unionofomega1reldiscrete} allows to conclude that $T$ itself is $\omega_1$-relatively discrete.

Since $T$ is compact, the Cantor-Bendixson rank of $T$ is a successor ordinal, therefore let us write $\rk(T)=\gamma+1$, for some ordinal $\gamma$. Hence, $T^{(\gamma)}$ is a finite subset of $T$; thus for each $t\in T^{(\gamma)}$ there are $s_0(t)\in I(T)$ and $F_t\subseteq \ims(t)$ such that the sets $W_{s_0(t)}^{F_t}$ are pairwise disjoint; in particular, $W_{s_0(t)}^{F_t}\cap T^{(\gamma)}=\{t\}$. Let us define the clopen set
\begin{equation*}
    S\coloneqq T\setminus \bigcup_{t\in T^{(\gamma)}}V_{s_0(t)}, 
\end{equation*}
which is $\wedge$-closed in $T$. Moreover, since $S\cap T^{(\gamma)}=\emptyset$, it follows that $\rk(S)<\rk(T)$. Hence, our assumption on $T$ yields us that $S$ is $\omega_1$-relatively discrete. 

Next, we have to take care of each $V_{s_0(t)}$; part of the difficulty here is that, while the $W_{s_0(t)}^{F_t}$ are disjoint, $V_{s_0(t)}$ might contain more than one point of $T^{(\gamma)}$. We start by handling the mutually disjoint sets $V_{s_0(t)} \setminus V_t$ (note that it might happen that $s_0(t)=t$, in which case the set is just empty). Thus fix $t\in T^{(\gamma)}$ and consider the tree $V_{s_0(t)}$. For every $\eta< \htte(t,V_{s_0(t)})$ we define $s_\eta(t)$ as the unique predecessor of $t$ that belongs to $\Lev_{\eta}(V_{s_{0}(t)})$. Then, for every $r\in I_\eta\coloneqq \ims(s_\eta(t))\setminus \{s_{\eta+1}(t)\}$ it holds that $V_r\cap T^{(\gamma)}=\emptyset$ (indeed, $V_r\subseteq V_{s_0(t)} \setminus V_t \subseteq W_{s_0(t)}^{F_t}$ and $t\notin V_r$).
Therefore, the Cantor-Bendixson rank of $V_r$ is smaller than $\rk(T)$. Hence, by the hypothesis $V_r$ is $\omega_1$-relatively discrete. Moreover, we have the disjoint union
$$\left(V_{s_0(t)} \setminus V_t\right) \cup\{t\} = (\hat{t}\cap V_{s_0(t)}) \cup \bigcup_{\eta< \htte(t, V_{s_0(t)})} \bigcup_{r\in I_\eta} V_r.$$

Hence, the last part of $T$ we have to take care of are the sets $V_t\setminus \{t \}$. Thus we fix $t\in T^{(\gamma)}$ and, for each $r\in \ims(t)$, we define $G_r\coloneqq V_r\cap \{s_0(p)\}_{p\in T^{(\gamma)}}$. Then, since $G_r$ is finite, $W_r^{G_r}$ is a clopen, $\wedge$-closed subset of $T$, which by definition does not intersect $T^{(\gamma)}$. Once more, our assumption on $T$ implies that $W_{r}^{G_r}$ is $\omega_1$-relatively discrete. Also, we clearly have
$$\left( V_t \setminus \{t\}\right) \setminus \bigcup_{t\neq p\in T^{(\gamma)}} V_{s_0(p)} = \bigcup_{r\in \ims(t)} W_r^{G_r}.$$

Finally, the previous considerations imply that $T$ can be written as follows
\begin{equation*}
    T= S  \cup \bigcup_{t\in T^{(\gamma)}}\bigcup_{r\in \ims(t)} W_r^{G_r} \cup \bigcup_{t\in T^{(\gamma)}}\bigcup_{\eta<\htte(t, V_{s_0(t)})}\bigcup_{r\in I_\eta} V_r \cup \bigcup_{t\in T^{(\gamma)}}(\hat{t}\cap V_{s_0(t)}),
\end{equation*}
where the above is a disjoint union. Moreover, we already saw that each clopen set $S$, $W_r^{G_r}$, and $V_r$ is $\omega_1$-relatively discrete. Since $\htte(T)\leq \omega_2$, for every $t\in T$ we have $\htte(t,T)<\omega_2$; therefore $\hat{t}$ has cardinality at most $\omega_1$. As $T^{(\gamma)}$ is finite, the subset $\bigcup_{t\in T^{(\gamma)}}(\hat{t}\cap V_{s_0(t)})$ has size $\omega_1$ as well. Therefore, Fact \ref{fact: unionofomega1reldiscrete} yields the assertion.
\end{proof}

\begin{theorem}\label{t: omega1discrete} Let $T$ be a scattered compact tree such that $\htte(T)\leq\omega_2$. Then $T$ is $\omega_1$-relatively discrete.
\end{theorem}

\begin{proof} Assume by contradiction that the result fails and let $\beta_0\leq \omega_2$ be the minimal height of a counterexample. Next, consider all scattered trees of height $\beta_0$ which fail to be $\omega_1$-relatively discrete and let $\alpha_0$ be the minimal value of $\rk(T)$ among all such trees $T$ ($\alpha_0$ is well defined since $T$ is scattered); finally, fix a scattered tree $T$ such that $\htte(T)= \beta_0$, $\rk(T)= \alpha_0$, and $T$ is not $\omega_1$-relatively discrete.

We shall now check that $T$ satisfies the condition of Lemma \ref{lem: S and discrete lev}, so we take a closed subset $S$ of $T$ which is $\wedge$-closed and such that $\rk(S)<\alpha_0$. Since $S$ is both closed and $\wedge$-closed, the subspace topology on $S$ is the coarse wedge topology of the tree $S$ (with the induced order), see \cite[Lemma 2.1]{S20}. Hence, $S$ is a scattered tree such that $\htte(S)\leq \beta_0$ and $\rk(S)< \alpha_0$. By minimality of $\beta_0$ and $\alpha_0$, it follows that $S$ is not a counterexample, thus $S$ is $\omega_1$-relatively discrete. Therefore, we can apply Lemma \ref{lem: S and discrete lev}, which gives us the contradiction that $T$ is $\omega_1$-relatively discrete and concludes the proof.
\end{proof}

Note that the above result cannot be reversed in ZFC. Indeed, assuming the continuum hypothesis, the dyadic tree of height $\omega+1$ has size $\omega_1$, thus it is $\omega_1$-relatively discrete. On the other hand, it contains a perfect subset, therefore it is not scattered.

For an \textit{r}-tree $T$ that satisfies $\htte(T)<\omega_2$, we define $R=\{t\in T\colon \cf(t)=\omega_1,\ \ims(t)\neq \emptyset\}$. By \cite[Theorem 3.2 (2)]{S20}, if $R\cap \Lev_{\alpha}(T)$ is $\omega_1$-relatively discrete for each $\alpha<\omega_2$ with $\cf(\alpha)=\omega_1$, then $T$ is Valdivia. Therefore Theorem \ref{t: omega1discrete} implies the following result.

\begin{corollary}\label{cor: scattered tree is Valdivia} Every scattered compact r-tree $T$ with $\htte(T)<\omega_2$ is Valdivia.
\end{corollary}

\begin{remark} When $\htte(T)>\omega_2$ the above corollary is clearly false, as it is simply witnessed by $[0,\omega_2]$ (and actually every Valdivia tree has height at most $\omega_2$, \cite[Theorem 3.2(1)]{S20}). On the other hand, it is conceivable that Corollary \ref{cor: scattered tree is Valdivia} is also valid for trees $T$ with $\htte(T)=\omega_2$ (especially because Theorem \ref{t: omega1discrete} is valid when $\htte(T)=\omega_2$); however, it is open if \cite[Theorem 3.2(2)]{S20} is also true when $\htte(T)=\omega_2$. 
\end{remark}

\begin{problem} Let $T$ be a scattered, compact $r$-tree with $\htte(T)=\omega_2$. Must $T$ be Valdivia?
\end{problem}

In the light of the previous results, it is natural to wonder when compact trees are scattered. Therefore, we now proceed to give a characterisation of scattered trees. Given an ordinal $\alpha$ we denote by $L(\alpha)$ the set of limit ordinals less than $\alpha$.
\begin{proposition}\label{prop: scattered iff}
    Let $T$ be a compact tree. The following are equivalent:
\begin{enumerate}
    \item\label{i: scattered} $T$ is scattered.
    \item\label{i: strong condition} For every $A\subseteq \bigcup_{\alpha\in L(\htte(T))} \Lev_{\alpha}(T)$, there exist $t\in A$, $s\in I(T)$ with $s<t$, and a finite set $F\subseteq \ims(t)$ such that $W_s^{F}\cap A=\{t\}$.
\end{enumerate}\end{proposition}

\begin{proof} Suppose that \eqref{i: scattered} holds and take a subset $A$ of $\bigcup_{\alpha\in L(\htte(T))} \Lev_{\alpha}(T)$. Since $T$ is scattered, the subset $A$ contains an isolated point $t\in A$. Hence there exist $s\in I(T)$ with $s<t$ and $F\subseteq \ims(t)$ such that $W_s^F\cap A=\{t\}$.

Conversely, assume \eqref{i: strong condition} and fix a closed subset $K$ of $T$. Let $A\coloneqq K\cap \bigcup_{\alpha\in L(\htte(T))} \Lev_{\alpha}(T)$. By assumption, there exist $t\in A$, $s\in I(T)$ with $s<t$, and a finite set $F\subseteq \ims(t)$ such that $W_s^F\cap A=\{t\}$. If $K\cap W_{s}^F\subseteq \hat{t}$, then $K\cap W_{s}^F$ has an isolated point, as $\hat{t}$ is scattered. Thus, $K$ has an isolated point as well. Therefore, we can select a point $u\in K\cap W_{s}^F \setminus \hat{t}$; we also take a maximal chain in $K\cap W_{s}^F \cap V_u$ (namely, a maximal chain in $K\cap W_{s}^F$, with minimum $u$). If such a chain is finite, then its maximum is an isolated point in $K\cap W_{s}^F$, so in $K$. Else, if the chain is infinite, then it admits an accumulation point. On the one hand, the accumulation point must belong to the closed set $K\cap W_{s}^F$; on the other one, it belongs to $A$, as it is on a limit level. But $W_s^F\cap A=\{t\}$, hence such an accumulation point must be $t$, which implies $u\leq t$. However, this contradicts $u\notin \hat{t}$ and finishes the proof.
\end{proof}

Similarly as in Proposition \ref{prop: scattered iff}, a characterisation of metrisable compact trees is also available, see, \emph{e.g.}, \cite[Lemma 5.22]{K20} for the more general result that $w(T)= \dens(T)= |I(T)|$ for every compact tree. For the sake of completeness, we give an alternative proof.
\begin{proposition}\label{p: carattmetrizz} For a compact tree $T$ the following are equivalent:
\begin{romanenumerate}
    \item $T$ is metrisable,
    \item $T$ is separable,
    \item $|I(T)|\leq \omega$.
\end{romanenumerate}
\end{proposition}

\begin{proof} (i)$\implies$(ii) is true for every compact topological space.\\
Next, we prove that (ii)$\implies$(iii). We first notice that if $T$ is separable, then $\htte(T)\leq \omega_1$. Indeed, otherwise Lemma \ref{lem: omega1 image} would imply that $[0,\omega_1]$ is a continuous image of $T$, which is impossible as $[0,\omega_1]$ is not separable. Therefore, if $\{t_n\}_{n\in\omega}$ is a countable dense subset of $T$, we have $\alpha_n \coloneqq \htte(t_n,T)<\omega_1$ for each $n\in\omega$. Thus we get $\sup_{n\in\omega}\alpha_n<\omega_1$ and $\htte(T)<\omega_1$. Now, suppose that $|I(T)|>\omega$. Since $\htte(T)$ is countable, by the pigeonhole principle, there exists a successor ordinal $\beta<\htte(T)$ such that $\Lev_{\beta}(T)$ is uncountable. Hence $\{V_t\}_{t\in\Lev_{\beta}(T)}$ is an uncountable pairwise disjoint family of nonempty open sets, which contradicts the separability of $T$.\\
Finally, we prove that (iii)$\implies$(i). Suppose that $|I(T)|\leq \omega$; then $\{V_t\}_{t\in I(T)}$ is a countable subbase of $T$. Thus $T$ is second countable, hence metrisable.
\end{proof}

After the proof that (ii)$\implies$(iii), it is natural to ask whether it's also possible to add the c.c.c. in the above characterisation. Recall that a topological space has the \emph{countable chain condition} (for short, is \emph{c.c.c.}) if every collection of nonempty, disjoint open subsets is at most countable. For information on Martin's Axiom for $\omega_1$ {\sf MA$_{\omega_1}$} we refer, \emph{e.g.}, to \cite[\S III.3]{Kunen}.

\begin{corollary}[{\sf MA$_{\omega_1}$}]\label{c: MA} A compact tree $T$ is c.c.c.\ if and only if $|I(T)|\leq \omega$.
\end{corollary}
\begin{proof} Every separable topological space is c.c.c.; thus we only need the converse implication. As above, we first notice that a c.c.c.\ compact tree $T$ satisfies $\htte(T)\leq\omega_1$. In fact, otherwise Lemma \ref{lem: omega1 image} would imply that $[0,\omega_1]$ is a continuous image of a c.c.c.\ space, which is a contradiction with the facts that the c.c.c.\ is preserved by continuous images and that $[0,\omega_1]$ is not c.c.c.. Consequently, every $t\in T$ satisfies $\htte(t,T)<\omega_1$ and it is easy to infer that $T$ is Corson, \cite[Theorem 2.8]{N2}. Finally, it is well known that {\sf MA$_{\omega_1}$} implies that c.c.c.\ Corson compacta are separable (see, \emph{e.g.}, \cite[p.~205]{CN chain cond}, or \cite[p.~207]{Fremlin}), and we are done.
\end{proof}

\begin{remark} If a Souslin tree exists, Corollary \ref{c: MA} is false (recall that the existence of a Souslin tree is consistent with ZFC, see \cite[\S III.5 and III.7]{Kunen}). Indeed, P. Nyikos showed that the completion of a uniformly $\omega$-ary Souslin tree is a non-separable,  hereditarily Lindel\"{o}f, Corson compact space in the coarse wedge topology (see \cite[Theorem 3.2]{N2}). Therefore, in particular, it is a non-separable c.c.c.\ Corson compact tree.
\end{remark}

We conclude this section with an explicit construction of a $1$-norming M-basis in $C(T)$, where $T$ is a metrisable compact tree. Notice that, when $T$ is metrisable, $C(T)$ is separable, so the existence of a $1$-norming M-basis is a classical result. Our point here is only to give an explicit (simple) formula. The construction is inspired by \cite[Example 5.37]{K20}, in which a shrinking M-basis is constructed in $C([0,\alpha])$ for every countable ordinal $\alpha$.

\begin{example}\label{ex: explicit norming M-basis}
Let $T$ be a metrisable compact tree. By Proposition \ref{p: carattmetrizz} the set $I(T)$ is countable. Therefore we can fix an injective enumeration $(\xi_n)_{n\in \omega}$ of $I(T)$ with the property that $\xi_0=0_T$. In order to define a $1$-norming M-basis $\{f_n,\mu_n\}_{n\in\omega}$ in $C(T)$ we make use of an induction argument. We first define, for each $n\in \omega$ the following sets
\begin{equation*}
F_{n+1}\coloneqq V_{\xi_{n+1}}\cap \{\xi_0,\dots,\xi_n\}
\end{equation*}
and
\begin{equation*}
G_{n+1}\coloneqq \hat{\xi}_{n+1}\cap \{\xi_0,\dots,\xi_n\}.
\end{equation*}
Now, as initial step of the induction argument we define
\begin{equation*}
    f_0\coloneqq \bone_{V_{\xi_0}} \,\,\,\mbox{ and }\,\,\, \mu_0\coloneqq \delta_{\xi_0}.
\end{equation*}
Suppose that $f_k \in C(T)$ and $\mu_k \in C(T)^*$ have been defined for every $k\leq n$, in such a way that $\langle \mu_{k},f_j\rangle =\delta_{j,k}$ for every $k,j\leq n$. We are going to define $f_{n+1}$ and $\mu_{n+1}$. 
Let
\begin{equation*}
    f_{n+1}=\bone_{V_{\xi_{n+1}}\setminus \bigcup_{t\in F_{n+1}}V_t} \,\,\,\mbox{ and } \,\,\, \mu_{n+1}=\delta_{\xi_{n+1}} - \delta_{\max G_{n+1}}.
\end{equation*}
Let us show that $\langle\mu_j, f_k\rangle =\delta_{j, k}$ for every $j,k\leq n+1$. By the induction hypothesis we may suppose that at least one among $j$ and $k$ is equal to $n+1$. If $k=j=n+1$, then clearly we get $\langle \mu_j, f_k\rangle =1$. We now suppose that $j=n+1>k$ and we distinguish two cases:
\begin{itemize}
    \item Assume that $\xi_k<\xi_{n+1}$. If there exists $k_1< k$ such that $\xi_k< \xi_{k_1}< \xi_{n+1}$, then $f_k(\xi_{n+1})=0$ as well as $f_k(\max G_{n+1})=0$. Otherwise, we have $f_k(\xi_{n+1})= 1= f_k(\max G_{n+1})$. In either case we get $\langle\mu_j, f_k\rangle=0$.
    \item If  $\xi_k \not< \xi_{n+1}$, then clearly $f_k(\xi_{n+1})= f_k(\max G_{n+1})=0$.
\end{itemize}
In the case when $k=n+1>j$, we argue similarly. Hence, we have $\langle\mu_j, f_k\rangle=0$ for each $j,k\leq n+1$.

Consequently, the family $\{f_n,\mu_n\}_{n\in\omega}$ is a biorthogonal system of $C(T)$. Moreover, it is immediate to realise by induction that $\{\bone_{V_t}\} _{t\in I(T)}\subseteq \Span\{f_n\}$ and $\{\delta_t\} _{t\in I(T)}\subseteq \Span\{\mu_n\}$; in particular $\Span\{\mu_n\}$ is $1$-norming. Finally, the equality $\bone_{T \setminus V_t}= \bone_{V_{\xi_0}} - \bone_{V_t}$ shows that also $\{\bone_{T \setminus V_t}\} _{t\in I(T)}\subseteq \Span\{f_n\}$. It readily follows that $\Span\{f_n\}$ is an algebra, which is therefore dense, due to the Stone-Weierstrass theorem. Thus, $\{f_n,\mu_n\}_{n\in\omega}$ is a 1-norming M-basis of $C(T)$, as desired.
\end{example}

\begin{remark} One can wonder if, in analogy with \cite[Example 5.37]{K20}, the M-basis in Example \ref{ex: explicit norming M-basis} is additionally shrinking. Let us notice that it is in general not the case. Indeed, if $T$ is not scattered, $C(T)$ is not Asplund, hence it admits no shrinking M-basis at all. On the other hand, if $T$ is scattered, then it is a standard fact that $T$ is countable \cite[Lemma VI.8.2]{DGZ}. Therefore, by the Mazurkiewicz--Sierpi\'nski theorem, $T$ is homeomorphic to a countable ordinal, and \cite[Example 5.37]{K20} applies. Finally, let us recall the following problem, that already appeared in the Introduction.
\end{remark}

\begin{problem}\label{pb: norming in tree} Let $T$ be a tree with $\htte(T)\leq\omega_1$. Must $C(T)$ have a norming M-basis?
\end{problem}


\end{document}